\documentclass[12pt, reqno]{amsart}
\usepackage{amssymb, amsthm, amsmath, amsfonts, bbm}
\usepackage{enumerate}
\usepackage{array}

\usepackage{color}

\usepackage{hyperref}
\usepackage{tipa}
\usepackage{esvect}

\usepackage{anysize}
\usepackage[pdftex]{graphicx}
\marginsize{2.3cm}{2.3cm}{2.3cm}{2.3cm}

\makeatletter
\DeclareFontFamily{U}{tipa}{}
\DeclareFontShape{U}{tipa}{m}{n}{<->tipa10}{}
\newcommand{\arc@char}{{\usefont{U}{tipa}{m}{n}\symbol{62}}}%

\newcommand{\arc}[1]{\mathpalette\arc@arc{#1}}

\newcommand{\arc@arc}[2]{%
  \sbox0{$\m@th#1#2$}%
  \vbox{
    \hbox{\resizebox{\wd0}{\height}{\arc@char}}
    \nointerlineskip
    \box0
  }%
}
\makeatother

\DeclareMathOperator{\per}{per}
\DeclareMathOperator{\length}{length}

\DeclareMathOperator{\diam}{diam}

\newtheorem*{def*}{Definition}

\newtheorem*{rem*}{Remark}
\newtheorem*{cor*}{Corollary}

\newtheorem{lem}{Lemma}
\newtheorem*{lem1'}{Lemma $\mathbf{1^\prime}$}
\newtheorem{theorem}{Theorem}%[section]
\newtheorem{lemma}{Lemma}

\newtheorem{corollary}[theorem]{Corollary}
\newtheorem*{conjecture*}{Conjecture}

\theoremstyle{definition}

\theoremstyle{remark}

\newcommand{\bt}{\begin{theo}}
\newcommand{\et}{\end{theo}}
\newcommand{\bl}{\begin{lem}}
\newcommand{\el}{\end{lem}}
\newcommand{\bc}{\begin{cor*}}
\newcommand{\ec}{\end{cor*}}
\newcommand{\br}{\begin{rem*}}
\newcommand{\er}{\end{rem*}}
\newcommand{\bp}{\begin{proof}}
\newcommand{\ep}{\end{proof}}
\newcommand{\bes}{\begin{ex}}
\newcommand{\ees}{\end{ex}}

\newcounter{fig}
\newcommand{\f}{\refstepcounter{fig} Figure \arabic{fig}. }

\begin{document}

\ifpdf
\DeclareGraphicsExtensions{.pdf, .jpg, .tif, .mps}
\else
\DeclareGraphicsExtensions{.eps, .jpg, }
\fi

\title[On the Lengths of Curves Passing through Boundary Points]{On the Lengths of Curves Passing through Boundary Points of a Planar Convex Shape}

\begin{abstract}
	We study the lengths of curves passing through a fixed number of points on the boundary of a convex shape in the plane.
	We show that for any convex shape $K$, there exist four points on the boundary of $K$ such that the length of any curve passing through these points is at least half of the perimeter of $K$.
	It is also shown that the same statement does not remain valid with the additional constraint that the points are extreme points of~$K$.
	Moreover, the factor $\frac12$ cannot be achieved with any fixed number of extreme points. We conclude the paper with few other inequalities related to the perimeter of a convex shape.
 \end{abstract}
%{One bound on perimeter of a convex shape}
	
%Upper bounds for the perimeter of convex shapes, with applications to convex hulls of random walks

\author{Arseniy~Akopyan}
\address{Arseniy~Akopyan, Institute of Science and Technology Austria (IST Austria), Am Campus~1, 3400 Klosterneuburg, Austria}
\email{akopjan@gmail.com}

\author{Vladislav~Vysotsky}
\address{Vladislav~Vysotsky, Imperial College London, St.\ Petersburg Department of Steklov Mathematical Institute}
\email{v.vysotskiy@imperial.ac.uk, vysotsky@pdmi.ras.ru}

\thanks{The work of A.A. is supported by People Programme (Marie Curie Actions) of the European Union's Seventh Framework Programme (FP7/2007-2013) under REA grant agreement n$^\circ$[291734].
	The work of V.V. is supported by People Programme (Marie Curie Actions) of the European Union's Seventh Framework Programme (FP7/2007-2013) under REA grant agreement n$^\circ$[628803].}


\maketitle

\subsection*{Introduction}
We study the lengths of curves passing through a fixed number of points on the boundary of a convex shape in the plane. 
All the curves considered are supposed to be rectifiable, i.e., have finite length.
By convex shapes we mean compact convex subsets of the plane, and we assume that a convex shape has a non-empty interior to avoid trivialities.

We first show that for any convex shape $K$, there exist four points on the boundary of $K$ such that the length of any curve passing through these points is at least half of the perimeter of~$K$; see Theorem~\ref{thm:four points on boundary}. It turns out that this statement is optimal: the lower bound $\frac12 \per K$ (where per $K$ denotes the perimeter of $K$) cannot hold for three points and we cannot exceed the factor $\frac12$ even by increasing the number of points.
Moreover, if we additionally require that these boundary points are extreme,
%if the curve is allowed to pass only through the extreme points of $K$, 
then it does not suffice to take four points and, in fact, the factor $\frac12$ cannot be achieved with any fixed number of extreme points; see Theorem~\ref{thm:no approximation}.
By convention, in the statements concerning a number of extreme points we do not require that the points be distinct; for example, we can choose five extreme points in a triangle.
% Such setting is natural if one considers curves passing through vertices of polygons.

We conclude the paper by considering curves whose convex hulls cover $K$, and, in particular, curves that pass through all extreme points of $K$. It is well-known that the length of such curves is at least $\frac12 \per K$, which explains the factor $\frac12$ appearing above. This consideration is related to the question of H.T.~Croft \cite{croft1969curves} on the minimal length of a curve such that its convex hull contains a unit disk. 

Our results can be regarded as upper estimates for the perimeter of a convex shape. The approach presented was motivated by our studies \cite{AkopyanVysotskyProbab} of a problem in probability theory concerning the trajectories of planar random walks whose convex hulls have atypically large perimeter. 
%Using a standard probabilistic argument, this problem easily reduces to the question of finding an appropriate upper bound for the %perimeter of the convex hull of a planar curve. 

%Similar questions were considered by Glazyrin and Mori{\'c}~\cite{GlazyrinMoric}, Langi \cite{langi2013ontheperimeters}, and Pinchasi %\cite{pinchasi2015ontheperimeter} who studied the total perimeter of disjoint polygons covered by a convex body. 

A similar question was considered by A.~Zirakzadeh~\cite{zirakzadeh1964aproperty}, who proved that any triangle with its vertices dividing the boundary of a convex shape $K$ into three arcs of equal lengths has perimeter at least $\frac{1}{2}\per K$. This was extended  by B.~Bollob{\'a}s \cite{bolobas1967anextremal} who proved that the perimeter of any inscribed $n$-gon with its vertices dividing the boundary of $K$ into $n$ equal arcs is at least $(1-2/n)\per K$, which is a tight bound for even $n$.

%The similar question was considered by Zirakzadeh~\cite{zirakzadeh1964aproperty}, who proved that any triangle inscribed in a closed %convex curve $\gamma$ such that its vertices divide $\gamma$ into three arcs of equal lengths has perimeter at least $\frac{1}{2}%\length \gamma$. Bollob{\'a}s \cite{bolobas1967anextremal} proved that the perimeter of any inscribed $n$-gon with its vertices %dividing $\gamma$ into $n$ arcs of equal lengths is at least $[n-2/n]\length \gamma$, which is a tight bound for even $n$.

The related works of A.~Glazyrin and F.~Mori{\'c}~\cite{GlazyrinMoric}, Z.~L\'angi \cite{langi2013ontheperimeters}, and R.~Pinchasi~\cite{pinchasi2015ontheperimeter} concern inequalities for the perimeters of a convex body and one or several disjoint polygons covered by the body. There is an impressive survey by P.~Scott and P.~W.~Awyong~\cite{ScottAwyong} of inequalities relating the perimeter and other characteristics (area, width, etc.) of convex shapes.

One may consider the results presented in the context of the open tour Euclidean traveling salesman problem of finding the shortest path connecting certain points in the plane.
This problem is usually considered in the asymptotic setting for large numbers of points.
Thus, L.~Fejes T\'oth \cite{fejes1940ubereinen} posed the question of finding the asymptotics of the length of this path through $n$ points in a unit square, and proved the lower bound $(4/3)^{1/4}\sqrt n+o(\sqrt n)$.
The best current upper bound $1.391 \sqrt n+o(\sqrt n)$ was obtained by H.~J.~Karloff~\cite{karloff1989howlong}. A probabilistic version of this problem was studied by J.~Beardwood, J.~H.~Halton, and J.~M.~Hammersley~\cite{beardwood1952theshortest}, who considered points in an arbitrary body in any dimension.

% The problem of finding the asymptotic of the length of the shortest path through a large number of points inside a shape was posed by L.~Fejes T\'oth \cite{fejes1940ubereinen} and
% S.~Verblunsky~\cite{verblunsky1951ontheshortest} and L.~Few~\cite{few1955theshortest};
 % and S.~Steinerberger~\cite{Steinerberger2015newbounds}.

% We do not know any references for the points taken inside another convex shapes.
%
% The present problem is very different since we consider a fixed number of points on the boundary and we work with general convex shapes. It is also very different from the Euclidean traveling salesman problem, where the main interest again is in various types of asymptotic behaviour in the number of points; moreover, the traveling salesman paths need to be closed.

\subsection*{Curves through four points on the boundary}
For any convex shape $K$ and any number $\varepsilon>0$, one can choose several points on the boundary $\partial K$ of $K$ such that the length of any curve passing through these points is at least $(1-\varepsilon) \per K$; we leave this statement without a proof.
The number of points required for this approximation depends on both $K$ and $\varepsilon$.
% If we approximate a boundary of a convex shape $K$ by a lot of points and require a curve $\gamma$ to pass through all of them, then length of $\gamma$ should be at least $(1-\varepsilon) \per K$, where $\varepsilon$ depends on how dense we put the points (we left this statement without proof).
It is clear that we cannot control $\varepsilon$ (uniformly in $K$) by increasing the number of points.
Indeed, fix $n$ and consider a thin $1\times L$ rectangle with $L\gg n$. Any $n$ points on its boundary can be connected by a curve of length exactly $L+n \approx \frac{1}{2} \per P$ (see Figure~\ref{fig: long qudrileteral}). Actually, $\varepsilon=1/2$ is the threshold.

\begin{center}
	\includegraphics{fig-poly-9}
	
	\f 	\label{fig: long qudrileteral}
	An $1\times L$ rectangle with a path of length $L+n$ through $n$ points.
\end{center}

\begin{theorem}
	\label{thm:four points on boundary}
	Let $K$ be a compact convex shape in the plane. There exist four points on the boundary of $K$ such that the length of any path connecting these points is at least $\frac 12 \per K$. 
\end{theorem}

	\begin{center}
		\parbox{7cm}{
		\begin{center}\vskip 0.1cm
			\includegraphics{fig-poly-4}
\vskip 0.1cm
			\f \label{fig: three points figures}
			The optimal configuration of three points in a thin lens.
		\end{center}}\,\,
		\parbox{7cm}{
	\begin{center}
			\includegraphics{fig-poly-411}
			
			\f \label{fig: another three points figures}		
			A non-optimal configuration of three points in a thin lens.
		\end{center}}
	\end{center}
	
% \begin{remark}
% 	\label{rem:three points on curve}
	It will not suffice to take three points. Indeed, consider a thin lens formed by two identical circular segments; see Figure~\ref{fig: three points figures}. 
	Let us show that the triple $\{a,m,b\}$ consisting of the endpoints $a$ and $b$ and the midpoint $m$ of either of the arcs maximizes the length of the shortest path through any triple of points on the boundary of the lens. Clearly, the length of the path $amb$ is less than the length of the arc.

Consider a triple of points $\{a',m', b'\}$ on the boundary of the lens, which are denoted in the order of increasing $x$-coordinates.
% , such that the shortest path through these points is longer than $amb$.
Since $|a a'| \le |a m'|$ and $\angle a'am'<60^\circ$ if the lens is thin enough, we see that $a'm'$ is not a longest side in the triangle $aa'm'$, and therefore $|am'|\ge|a'm'|$ (and equality holds only if $a'$ coincides with $a$). By the same argument, we have $|m'b|\ge|m'b'|$. Hence the path $a'm'b'$ (which is not required to be a shortest path connecting $a'$, $m'$, $b'$) is not longer than $am'b$. It is left to note that the path $am'b$ has maximal length if and only if $m'$ is the midpoint of one of the arcs (apply Lemma~\ref{lem:triangle fixed angle} for the angle $\varphi = \angle am'b$, which does not depend on $m'$). 
	
	In the proof of Theorem \ref{thm:four points on boundary} we will use the following statement, which, in our opinion, is interesting by itself.
% \end{remark}

% Our proof of Theorem~\ref{thm:four points on boundary} rests on the following statement.

\begin{theorem}
	\label{thm:double perimeter}
	Let $K$ be a compact convex shape in the plane, and let $ab$ be one of its diameters.
	Suppose that the perpendicular bisector of $ab$ intersects the boundary of $K$ at two points $c$ and~$d$.
	Then
	\begin{equation*}
		\per K < 2|ab|+2|cd|.
	\end{equation*}
	%where $\per K$ denotes the perimeter of $K$.
\end{theorem}

\begin{proof}
	The diameter $ab$ divides the boundary of $K$ into two parts $K_c$ and $K_d$ that contain the points $c$ and $d$, respectively. 	We will show that if $c \notin ab$, then the length of $K_c$ is greater than $|ab|+2|co|$, where $o$ is the midpoint of $ab$. Since at least one of the points $c$ and $d$ does not lie on $ab$, this inequality with the analogous inequality for $K_d$ implies the statement of the theorem.

	% , and that the inequality is strict if $c \notin ab$.
	 % because at least one of the points $c$ and $d$ does not lie on $ab$
% The case $c \in ab$ is trivial, and we may assume without loss of generality that $c \notin ab$.
% In fact, since $c \neq d$, at least one of the points $c$ and $d$ is not in $ab$, so at least one of the two inequalities is strict.

	\begin{center}
		\includegraphics{fig-poly-1}
		\item 
		\f \label{fig: circles}
		Illustration for the proof of Theorem~\ref{thm:double perimeter}.
	\end{center}
	
	Let us construct two circles of radius $ab$ centered at $a$ and $b$; see Figure~\ref{fig: circles}. 
	It is clear that $K$ lies in the intersection of the corresponding disks. We draw a support line to $K$ at the point~$c$ and denote its intersection points with the circles centered at $a$ and $b$ by $q$ and~$p$, respectively. Without loss of generality we assume that $q$ is farther from (not closer to) $ab$ than~$p$. Then $K_c$ lies inside the region bounded by the closed convex curve $apqba$ ($ap$ and~$qb$ are circular arcs), hence the length of $K_c$ does not exceed the length of $apqb$. We seek to bound the length of $apqb$.
	
	Denote by $p_1$ and $q_1$ the points of intersection of $pq$ with the respective perpendiculars to~$ab$ at the points $a$ and $b$; see Figure~\ref{fig: circles}.
	Since $2|co|=|p_1a|+|q_1b|$, it suffices to show that the length of $apqb$ is less than $|ab|+|p_1a|+|q_1b|$. 
	
	The line through the point $a$ parallel to $pq$ intersects the arc $qb$ and the segment $q_1b$ at the points $l$ and $m$, respectively. Denote by $k$  the point on $p_1q_1$ such that the quadrilaterals $ap_1kl$ and $lkq_1m$ are parallelograms.
	
	The length of the arc $lb$ satisfies $|\arc{\it lb}| = |ab| \cdot \angle(bal)$ and we have that $|bm|=|ab|\tan(\angle bal)$.
	Hence, since $\tan(x)\geq x$, we have
	\begin{equation}
		\label{eq:tangent inequality}
		|bm|\geq |\arc{\it lb}|.
	\end{equation}
	
	Since the closed curve $apqla$ is convex, the length of $apql$ is less than the length of the path $ap_1kl$.  	Since  $|kl|=|p_1a|=|q_1m|$ and $|p_1k|=|al|=|ab|$, we have
	\begin{equation}
		\label{eq:convex inequality}
		\length{(apql)}<|ap_1|+|q_1m|+|ab|.
	\end{equation}
	Combining \eqref{eq:tangent inequality} and \eqref{eq:convex inequality}, we obtain the required inequality for the length of the path $apqb$.
\end{proof}

\begin{proof}[Proof of Theorem \ref{thm:four points on boundary}]
	As in Theorem~\ref{thm:double perimeter}, denote by $ab$ a diameter of $K$, by $o$ the midpoint of $ab$, and by $c$ and $d$ the points of intersection of the perpendicular bisector of $ab$ with the boundary of $K$; see Figure~\ref{fig: four points figures}. Without loss of generality we can assume that $|co|\geq|do|$ and $|ao|=1$.
	
	\begin{center}
		\includegraphics{fig-poly-2}
		
		\f 	\label{fig: four points figures}
		Illustration for the proof of Theorem~\ref{thm:four points on boundary}.
	\end{center}
	
	It is clear that the shortest path connecting the points $a$, $b$, $c$, $d$ is $acdb$ (if $|ad|\geq|cd|$) or $cadb$ (if $|ad|<|cd|$). In the former case, the length of the path is greater than $|ab|+|cd|$, which exceeds $\frac12 \per K$ by Theorem~\ref{thm:double perimeter}.
	
	In the latter case, we note that $|cd|>|ad|\geq 1$. For a fixed $|cd|$, $|ad|+|ac|$ attains its minimum value when $acd$ is an isosceles triangle.
	Therefore,  $|ad|+|ac|>\sqrt{5}>2.2$. Since $|ad|\geq 1$, we conclude that  the length of the path $cadb$ is at least $3.2>\pi$.
	It remains to use the fact that the perimeter of a convex set of diameter $2$ is at most $2\pi$. (This follows from the Crofton formula \eqref{eq: Crofton}.)
\end{proof}

\subsection*{Curves through extreme points.}
%It is interesting that for the case if we aloud to choose only extreme points of $K$ the situation changes, the halfperimeter cannot %be achieved for any $n$.
We now consider curves that are required to pass through extreme points of convex shapes.
Let us recall that a point of a convex shape $K$ is called \emph{extreme} if it does not belong to any open line segment with end points in $K$. In the case that $K$ is a convex polygon, the set of its extreme points coincides with its vertices.

\begin{theorem}
	\label{thm:no approximation}
	For any $n \ge 2$, there exists a convex shape $K_n$ such that any $n$ extreme points of $K_n$ can be connected by a path of length less than $\frac12 \per K_n$.
\end{theorem}

Since any convex shape can be approximated by a convex polygon, we obtain the following corollary.
\begin{corollary}
	\label{cor::no approximation}
	For any $n \ge 3$, there exists a convex polygon $P_n$ such that any $n$ vertices of $P_n$ can be connected by a path of length less than $\frac12 \per P_n$.
	%For any $n \ge 2$, there exists a convex polygon $P_n$ such that for any $n$ vertices of $P_n$ %there is a path connecting them with the length less than $\frac12 \per P_n$.
\end{corollary}
\begin{proof}[Proof of Theorem~\ref{thm:no approximation}]
Let $E_k$ be a half of an elongated ellipse with semi-axes $1$ and $k$, where $k$ is sufficiently small (to be chosen later), bisected through its major axis $ab$ of length two, as shown in Figure~\ref{fig:moving point}. 
% Since $\per K$ can be approximated as closely as necessary by the perimeter of a convex polygon with vertices on the arc $\arc{ab}$,
Since the set of extreme points of $E_k$ is the arc $\arc{ab}$, it suffices to prove that for any $n$ points $m_1, \dots, m_n$ on the arc enumerated in the direction from $a$ to $b$, we have
\begin{equation} \label{eq:length n <}
|m_1 m_2| + \dots + |m_{n-1} m_n| < \frac12 \per E_k.	
\end{equation}
Note that we do not require that the polygonal path $m_1m_2\dots m_n$ be the shortest among the paths connecting the points $m_i$, and 
although it is not hard to verify this statement, we will not prove it.
%It is not hard to show that the polygonal line $m_1m_2\dots m_n$ is the shortest among the paths connecting points $m_i$. This is not %actually needed because our goal is only to show that the points $m_i$ can be connected by a path of length less that $\frac12 \per %E_k$.

We use the following lemma, which is proved right after the proof of Theorem~\ref{thm:no approximation}.
\begin{lemma}
	\label{lem:adding vertex}
	For any three points $m_{i-1}$, $m_i$, $m_{i+1}$ on the arc $\arc{ab}$ 
	%(enumerated in the direction from $a$ to $b$),
	(such that $m_i \in \arc{m_{i-1} m_{i+1}}$),
	\begin{equation} \label{eq:small sum of 2}
	|m_{i-1} m_i|+|m_i m_{i+1}| - |m_{i-1}m_{i+1}| \leq 2 \sqrt{1 + k^2} - 2.
	\end{equation}
\end{lemma}

Applying the lemma $n-2$ times to the triples of points $m_1$, $m_i$, $m_{i+1}$, we find that the length of the polygonal line $m_1 \dots m_n$ satisfies
\begin{equation} \label{eq:small sum of n}
|m_1 m_2| + \dots + |m_{n-1} m_n| \le 2(n-2)(\sqrt{1 + k^2} - 1) +2
\end{equation}
since $|m_1 m_n| \leq |ab|$.

On the other hand, the perimeter of an elongated ellipse with semi-axes $1$ and $k$ has the asymptotics $4+ 2{k^2} \log \frac{1}{k} + O({k^2})$ as $k \to 0$, which corresponds to the first two terms of the so-called Cayley series (see \cite[Ch.~III.78]{Cayley} or~\cite[Eq.~8.114.3]{GradshteinRyzhik}). Hence
\begin{equation}
	\label{eq:half perimeter}
	\frac12 \per E_k = 2 + \frac{1}{2} k^2 \log \frac{1}{k} + O({k^2}), \qquad k \to 0.
\end{equation}

The statement of Theorem~\ref{thm:no approximation} now follows by~\eqref{eq:length n <},~\eqref{eq:small sum of n}, and~~\eqref{eq:half perimeter} since $n$ is fixed and $\log\frac{1}{ k}\to \infty $ as $k \to 0$: we choose $k$ small enough and put $K_n:=E_k$. 
\end{proof}

% We will need the following simple fact.
% \begin{lemma}
% 	\label{lem:diagonals and two sides}
% 	For any convex quadrilateral, the sum of the lengths of its diagonals exceeds the sum of the lengths of either pair of opposite edges.
% \end{lemma}
% \begin{proof}
% 	Denote the vertices of the quadrilateral by $a$, $b$, $c$, $d$ and the intersection of the diagonals by $o$.
% 	By the triangle inequality applied for the triangles $abo$ and $cdo$, we obtain
% 	\begin{equation*}
% 		|ao|+|bo|>|ab|
% 		\text{ and }
% 		|co|+|do|>|cd|
% 		\Rightarrow
% 		|ac|+|bc|>|ab|+|cd|.
% 	\end{equation*}
% \end{proof}

\begin{proof}[Proof of Lemma~\ref{lem:adding vertex}]
	
	We claim that 
	\begin{equation}
		\label{eq: m to b}
		|m_{i-1} m_i|+|m_i m_{i+1}| - |m_{i-1}m_{i+1}| < |m_{i-1} m_i|+|m_i b| - |m_{i-1}b|; 
	\end{equation}
	see Figure~\ref{fig:moving point}. This is equivalent to 
	\begin{equation*}
		%\label{eq:m to b quadrileteral}
		|m_i m_{i+1}| + |m_{i-1}b| <|m_i b| + |m_{i-1}m_{i+1}|,
	\end{equation*}
    which follows from the fact that for any convex quadrilateral, the sum of the lengths of its diagonals exceeds the sum of the lengths of either pair of opposite sides. This fact holds by the triangle inequality applied to the two triangles formed by the intercepts of the diagonals and the corresponding side.
	
	% directly by Lemma~\ref{lem:diagonals and two sides} applied to the quadrilateral $m_{i}m_{i+1}bm_{i-1}$.
	
		\begin{center}
		\includegraphics{fig-poly-5}
		
		\f \label{fig:moving point}
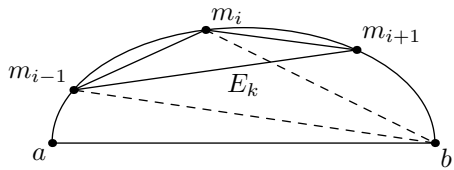
		Illustration for the proof of Lemma~\ref{lem:adding vertex}.
	\end{center}	
	
	Analogously to \eqref{eq: m to b}, we have
	\begin{equation}
		\label{eq: m to a}
		|m_{i-1} m_i|+|m_i b| - |m_{i-1}b| < |a m_i|+|m_i b| - |a b|.
	\end{equation}
	
 	% First, let us show, that if we fix the points $m_{i-1}$ and $m_i$ and move the point $m_{i+1}$ along $K$ towards $b$, then the value of $|m_i m_{i+1}| - |m_{i-1}m_{i+1}|$ increases, see Figure~\ref{fig:moving point}.
%
% First, let us show that if we fix the points $m_{i-1}$ and $m_i$ and move the point $m_{i+1}$ along $K$ towards $b$, then the value of $|m_i m_{i+1}| - |m_{i-1}m_{i+1}|$ increases, see Figure~\ref{fig:moving point}.
%
% Indeed, suppose the point $m_{i+1}$ moves with the velocity $\vv v$. Then
% \[
% |m_im_{i+1}|'= |\vv v| \cdot \cos(\angle(\vv{m_im_{i+1}}, \vv v))
% \text{ and }
% |m_{i-1}m_{i+1}|'= |\vv v| \cdot \cos(\angle(\vv{m_{i-1}m_{i+1}}, \vv v)).
% \]
%  Since $K$ is the boundary of a convex set, we have $\angle(\vv{m_im_{i+1}}, \vv v)>\angle(\vv{m_{i-1}m_{i+1}}, \vv v)>0$, hence $|m_{i-1}m_{i+1}|'<|m_{i}m_{i+1}|'$.
%
% Analogously, $|m_{i-1} m_i| - |m_{i-1}m_{i+1}|$ increases as we move the point $m_{i-1}$ towards $a$. Thus,
% \begin{equation*}
% |m_{i-1} m_i|+|m_i m_{i+1}| - |m_{i-1}m_{i+1}| \le
% |am_i|+|m_ib| - |ab|.
% \end{equation*}

It remains to use the fact that $|am_i|+|m_ib|$ reaches its maximum if $m_i$ is the midpoint $m$ of the arc $\arc{ab}$: then \eqref{eq:small sum of 2} follows since $|am|+|mb| - |ab|=2 \sqrt{1 + k^2} - 2$. 

Indeed, if the maximum is attained at some other point $m'$, then the ellipse $E'$ with foci $a$ and $b$ and major axis of length $|am'|+|bm'|$ touches the half-ellipse $E_k$ (with semi-axes $1$ and $k$) at two points, namely, $m'$ and its symmetric image about the minor axis of $E_k$. By taking the symmetric image of $E_k$ about its major axis, we obtain the complete ellipse inscribed in $E'$ and touching it at four points, which is impossible for two conic curves.  % since this is the only point on the arc where the tangent line coincides with the exterior bisector of the angle $am_ib$.
\end{proof}

The positive result here is that any fraction less than half of the perimeter can be reached by increasing the number of vertices.

\begin{theorem}
	\label{thm:1/2-e approximation}
	For any $\varepsilon>0$ there exists a positive integer $n$ such that for any convex shape $K$ one can choose $n$ extreme points that cannot be connected by a curve of length less than~$\frac{1-\varepsilon}{2} \per K$.
\end{theorem}

For the proof we will need the following statement.
\begin{lemma}
	\label{lem:triangle fixed angle}
	For any triangle $abc$ with the angle $\angle bac = \varphi$, we have
	\begin{equation}
		\label{eq:sum of sides}
		\frac{|bc|}{|ab|+|ac|} \geq \sin \frac{\varphi}{2}.
	\end{equation}
\end{lemma}
\begin{proof}
	Note that $|ab|\sin \frac{\varphi}{2}$ and $|ac|\sin \frac{\varphi}{2}$ are the distances to the angle bisector of the angle $\angle bac$ from $b$ and $c$, respectively. The sum of these distances is at most $|bc|$.  
% For fixed points $b$ and $c$ the locus of points $a$ such that $\angle bac = \varphi$ is the union of two circular arcs.
% It is clear that for points $a$ belonging to each of the arcs, the value $|ab|+|ac|$ reaches its maximum when $a$ is the most distant from $bc$.
% In that case the triangle $abc$ is isosceles and we have an equality in~\eqref{eq:sum of sides}.
\end{proof}

\begin{proof}[Proof of Theorem~\ref{thm:1/2-e approximation}]
	Choose $n$ such that $\cos\frac{\pi}{n} > 1-\varepsilon$.
	For each $1 \leq i \leq n$, choose an extreme point $v_i$ of $K$ that admits a support line $\ell_i$ through it with the outer normal vector at the angle $2\pi\frac{i}{n}$ with some fixed direction (see Figure \ref{fig: circuminscribed construction}).	
	%such that there exists the support line $\ell_i$ at $v_i$ with the outer normal vector at the angle $2\pi\frac{i}{n}$ with some %fixed direction (see Figure \ref{fig: circuminscribed construction}).
	If there are two such points choose either of them; some extreme points can correspond to several~$i$.
	
	Now let us show that perimeter of the convex polygon $V=v_1v_2\dots v_{n}$ is at least \mbox{$(1-\varepsilon) \per K$}.
	Denote by $o_i$ the intersection of the support lines $\ell_i$ and $\ell_{i+1}$ (we assume that $\ell_{n+1}=\ell_1$ and $v_{n+1}=v_1$). 
	Note that the part of the perimeter of $K$ lying between $v_i$ and $v_{i+1}$ has length at most $|v_io_i|+|o_iv_{i+1}|$, which by Lemma~\ref{lem:triangle fixed angle} is at most $\frac{|v_iv_{i+1}|}{\cos \pi/n}$.
	Applying this inequality for all the arcs $\arc{v_iv_{i+1}}$ of the perimeter of $K$, we obtain the inequality $\per K \leq \frac{\per V}{\cos \pi/n}$. Therefore, $(1-\varepsilon)\per K < \per V$.
	
	By Theorem~\ref{thm:half of perimeter}, which is an independent statement presented below, the length of any curve passing through all vertices of $V$ should be at least $\frac{1}{2}\per V>\frac{1-\varepsilon}{2}\per K$.
\end{proof}

\begin{center}
	\includegraphics{fig-poly-11}
	
	\f 	\label{fig: circuminscribed construction}
	Illustration for the proof of Theorem~\ref{thm:1/2-e approximation}.
\end{center}

Note that in \cite{schneider1971zwei} R.~Schneider applied a similar construction for his solution of the problem of L.~Fejes T\'oth on the $n$-gon of the maximum (minimum) perimeter inscribed (resp., circumscribed) in a convex shape \cite[p. 39]{toth1953lagerungen}.

\subsection*{Barriers for convex shapes}
The problem of finding a shortest curve whose convex hull covers a unit disk was posed by H.~T.~Croft in \cite{croft1969curves} and solved by V.~Faber, J.~Mycielski, and P.~Pedersen in~\cite{faber1984ontheshortest}.
Following \cite{dumitrescu2014opaque}, let us call such a curve a \emph{barrier}.
Not much is known if instead of the unit disk we consider a general convex shape.
In \cite{faber1986theshortest} V.~Faber and J.~Mycielski give examples of plausibly optimal barriers for regular $n$-gons, $n\leq 6$, a halfdisk, and a certain parallelogram.

The following statement is widely known and even mentioned to be ``folklore.''
\begin{theorem}[See \cite{dumitrescu2014opaque} or \cite{faber1984ontheshortest}]
   \label{thm:half of perimeter}
    Let $\gamma$ be a curve such that its convex hull covers a planar convex shape $K$.
	Then 
\[\length \gamma \geq \frac12 \per K.\]
\end{theorem}

%Let us show, that if we require the condition that $\gamma$ passes through all extreme points of the shape $K$ then the stronger %inequality holds.
Let us prove a similar statement.

\begin{theorem}
   \label{thm:length of curve}
     Let $K$ be a convex shape on a plane, and let $\gamma$ be a curve passing through all extreme points of $K$. Then
	\[
	\length \gamma \geq \per K - \diam K.
	\]
\end{theorem}	

% Similar quantities appeared in Glazyrin and Mori{\'c}~\cite{GlazyrinMoric}, Langi \cite{langi2013ontheperimeters}, and Pinchasi \cite{pinchasi2015ontheperimeter} who studied the perimeter of one or several disjoint polygons covered by a convex body.
\begin{proof}
	% [Proof of Theorem~\ref{thm:length of curve}]
	Let $a$ and $b$ be the first and the last points (with respect to any parametrization of~$\gamma$) of intersection of~$\gamma$ and $\partial K$; see Figure~\ref{fig: gamma}.
	Define the closed curve $\gamma'$ formed by the part of $\gamma$ between points $a$ and $b$ and the line segment $ab$.
	Since $\gamma'$ passes through all the extreme points of $K$, its convex hull covers $K$.
	Since $|ab|\leq \diam K$, the claim will follow if we show that the length of $\gamma'$ is at least~$\per K$.

\begin{center}
	\includegraphics{fig-poly-10}
	
	\f 	\label{fig: gamma}
	Illustration for the proof of Theorem~\ref{thm:length of curve}.
\end{center}	
	
%	This is classical statement, and the simplest proof of it based on on the Crofton formula from integral geometry (see %e.g.,~Tabachnikov\cite{Tabachnikov2005}):
	Let us use the Crofton formula from integral geometry (see e.g.,~S.~Tabachnikov \cite{Tabachnikov2005}):
	\begin{equation} \label{eq: Crofton}
		 \length (\gamma') = \frac12 \iint \limits_{\mathbb{S}^1 \mathbb{R_+}}n_{\gamma'}(\phi, p) d\phi dp,
		 \,\,\,\,\,\,\,
		 \per (K) = \frac12 \iint \limits_{\mathbb{S}^1 \mathbb{R_+}} n_{\partial K}(\phi, p) d\phi dp, 
	\end{equation}
where $n_\nu(\phi, p)$ denotes the number of intersections of a curve $\nu$ with the line perpendicular to the direction $\phi$ passing at the distance $p$ from the origin.
We have that $n_{\gamma'}(\phi, p) \geq n_{\partial K}(\phi, p)$ for almost every pair $(\phi, p)$.
Indeed, each line intersecting $K$ intersects $\gamma'$.
Since $\gamma'$ is closed, almost every line that intersects it has at least two points of intersection with $\gamma'$, while almost every line that intersects $\partial K$ has exactly two points of intersection with $\partial K$ since~$K$ is convex.
\end{proof}

  The authors believe that the following generalization is true.
\begin{conjecture*} \label{con:length of curve}
    Let $\gamma$ be a curve such that its convex hull covers a planar convex shape $K$.
	Then 
	\[
	\length \gamma \geq \per K - \diam K.
	\]
\end{conjecture*}

Note that the proof of Theorem~\ref{thm:length of curve} does not work if the convex hull of $\gamma$ does not cover~$K$ or if the distance between the endpoints is greater than $\diam K$.

\begin{center}
	\includegraphics{fig-poly-8}
	
	\f 	\label{fig: two quadrileterals}
	Example showing that $\per K - \diam K$ is not inclusion monotone.
\end{center}

It looks plausible that the function $\per K - \diam K$ is inclusion monotone.
But this is not true.
A counterexample is shown in Figure~\ref{fig: two quadrileterals}.
Let $abcd$ and $abcd'$ be deltoids containing an equilateral triangle $abc$ of the side length $1$, with their axes of respective lengths $1$ (the angle~$d$ equals $150^\circ$) and $2/\sqrt{3}$ (the angle $d'$ equals $120^\circ$).
The diameters of the deltoids are $bd$ and $bd'$. Then the value of $\per K-\diam K$ equals $(2+4\sin 15^\circ)-1\approx 2.035$ for the quadrilateral $abcd$ and $(2+2/\sqrt3)-2/\sqrt3=2$ for $abcd'$.

\subsection*{Acknowledgments} We wish to thank the three anonymous reviewers of the journal ``The American Mathematical Monthly'' for their comments and many suggestions for improving the paper.

 % \bibliographystyle{abbrv}
 % \bibliography{../perimeter}

\vskip 1cm
\end{document}